\documentclass[oneside,11pt]{amsart}
\usepackage[active]{srcltx}
\usepackage{amsmath,amssymb}
\newcommand{\vertiii}[1]{{\left\vert\kern-0.25ex\left\vert\kern-0.25ex\left\vert #1
    \right\vert\kern-0.25ex\right\vert\kern-0.25ex\right\vert}}
\usepackage{amsmath, amsfonts,amsthm,times,graphics}

 \makeatletter
\renewcommand*\subjclass[2][2000]{%
  \def\@subjclass{#2}%
  \@ifundefined{subjclassname@#1}{%
    \ClassWarning{\@classname}{Unknown edition (#1) of Mathematics
      Subject Classification; using '1991'.}%
  }{%
    \@xp\let\@xp\subjclassname\csname subjclassname@#1\endcsname
  }%
}
 \makeatother

\newtheorem{theorem}{Theorem}[section]

\newtheorem{lemma}[theorem]{Lemma}
\newtheorem{corollary}[theorem]{Corollary}
\newtheorem{proposition}[theorem]{Proposition}
\theoremstyle{definition}

\newtheorem{example}[theorem]{Example}

\numberwithin{equation}{section}
\newcommand{\abs}[1]{\lvert#1\rvert}

\newcounter{minutes}
\divide\time by 60
\newcounter{hours}
\multiply\time by 60 \addtocounter{minutes}{-\time}

\begin{document}

\title{Schwarz lemma for harmonic mappings between Riemann surfaces}


\author{David Kalaj}
\address{Faculty of Natural Sciences and Mathematics, University of
Montenegro, Cetinjski put b.b. 81000 Podgorica, Montenegro}
\email{davidk@ac.me}

\def\thefootnote{}
\footnotetext{ \texttt{\tiny File:~\jobname.tex,
          printed: \number\year-\number\month-\number\day,
          \thehours.\ifnum\theminutes<10{0}\fi\theminutes }
} \makeatletter\def\thefootnote{\@arabic\c@footnote}\makeatother

\begin{abstract}
We prove a Schwarz type lemma for harmonic mappings between the unit and a geodesic line in a Riemenn surface.
\end{abstract}
\maketitle
\tableofcontents

\section{Introduction}

Schwarz lemma is one of cornerstones of complex analysis. It is originally formulated for homolomorphic mappings of the unit disk onto itself. One of recent and important extension is due to Osserman \cite{oser}. Complex harmonic mappings, which contains holomorphic mappings,  play a substantial role in complex analysis, in particular in connection with the conformal parameterization of minimal surfaces. Concerning the Schwarz lemma several sharp results have been established.

Let $f$ be a real-valued function harmonic in the unit disk $\mathbf{U}$ into the interval $I_r=[-r,r]$. Then \begin{equation}\label{schwarz}| f (z)-\frac{1-|z|^2}{1+|z|^2}f(0)| \le \frac{ 4r}{\pi}\arctan |z|,\end{equation} and this inequality is sharp for
each point $z \in \mathbf{U}$ (See  \cite[Theorem~3.6.1]{pavlovic} and \cite[Theorem~6.24]{ABR} for $f(0)=0$). Furthermore, if $f$ is a complex harmonic mapping of the unit disk into the disk $\mathbf{U}_r$ so that $f(0)=0$, then the function $g(z)=\left<f(z),e^{is}\right>$ is a real harmonic mapping of the unit disk onto $I_r$, and therefore it satisfies the sharp bound \eqref{schwarz}. This infer that the inequality \eqref{schwarz} continuous to hold for complex harmonic mappings. The bound is sharp everywhere (but is attained
only at the origin) for univalent harmonic mappings $f$ of $\mathbf{U}$ onto itself with
$f (0) = 0.$
Further by using \eqref{schwarz} it can be derived the following result of Colonna in \cite{col}:

\begin{equation}\label{eq}
|f(z)-f(w)|\le \frac{4r}{\pi}d_h(z,w),
\end{equation}
where $d_h$ is the hyperbolic distance of between $z$ and $w$ on the unit disk, i.e. $$d_h(z,w)=\tanh^{-1}\left|\frac{z-w}{1-\bar{z}w}\right|.$$ The result of Colonna has been improved for real harmonic mappings by the author and Vourinen in \cite{vuko}. For related result we refer also to \cite{mm} and \cite{kazhu}.

The aim of this paper is to consider the Schwarz lemma for harmonic mappings between domains of Riemann surfaces.

Let be $D,\Omega\subset \mathbf{C}$ a domain and $(D,\rho)$ be a Riemann surface with the conformal metric $\rho\in C^1(D)$. Then a $C^2$ mapping $f:\Omega\to D$ is called harmonic if
\begin{equation}\label{eqeq}f_{z\bar z}+\partial_w \log\rho^2 \circ f\cdot  f_z f_{\bar z}=0, \ \  z\in \Omega.\end{equation}

Moreover a mapping $f\in C^2$ satisfies \eqref{eqeq} if and only if the function

$$a(z)=\rho^2(f(z))f_z \overline f_z $$ is holomorphic and the expression $a(z)dz^2$  is called the Hopf differential of $f$ and it is denoted by $\mathrm{Hopf}(f)$. It is a quadratic differential defined in $\Omega$. The mappings that satisfy the equation \eqref{eqeq} are stationary points of the energy integral defined by the equation $$\mathcal{E}_\rho[f]=\int_{\Omega} \rho^2(f(z)) (|f_z|^2+|f_{\bar z}|^2)dxdy,$$ for smooth mappings $f:D\to \mathbf{C}$ satisfying a boundary condition $f|_{\partial \Omega}=g$.

The main theorem of this paper is

\begin{theorem}\label{copo1}
Assume that $\rho$ is a metric defined in a domain $D$ and assume that $\gamma$ is a geodesic line contained in a geodesic disk $D_\rho(\omega,r)\subset D$ so that $\omega\in \gamma$.
Let $f:\mathbf{U}\to \gamma$ be a $\rho-$harmonic mapping with $f(0)=\omega$. Then for $z,z_1,z_2\in\mathbf{U}$ we have the sharp inequalities
\begin{equation}
d_\rho(f(z), f(0))\le \frac{4}{\pi}\arctan |z|,
\end{equation}

\begin{equation}\label{thosh}\rho(f(z))|Df(z)|\le \frac{4r}{\pi}\frac{1}{1-|z|^2}\end{equation}
and
\begin{equation}\label{sos}d_\rho(f(z_1),f(z_2)\le \frac{4r}{\pi}d_h(z_1,z_2).\end{equation}
\end{theorem}

We say that $\Omega$ is a hyperbolic domain, if there is a conformal bijection $k:\mathbf{U}\to \Omega$. Since conformal mappings are isometries of hyperbolic distance, and since a mapping $f:\Omega\to D$ is harmonic if and only if $f\circ k:\mathbf{U}\to D$ is harmonic we get the following corollary:

\begin{corollary}
Assume that $f$ is a $\rho-$harmonic mapping of the hyperbolic domain $\Omega$ and a geodesic line $\gamma\subset D_\rho(\omega,r)$. Then the sharp inequality
\begin{equation}\label{sosi}d_\rho(f(z_1),f(z_2)\le \frac{4r}{\pi}d_h(z_1,z_2), \end{equation} hold true for $\ \ \ z_1,z_2\in\Omega$, where $d_h$ is the hyperbolic  distance on $\Omega$.
\end{corollary}

\begin{corollary}\label{qpo}
Assume that $\rho$ is a radial metric defined in $|w|<R$. And let $f:\mathbf{U}\to D_\rho(0,r)$ be a real $\rho-$harmonic function so that $f(0)=0$ and $$r<\int_0^R \rho(s)ds. $$
Then for $|z|<1$
$$\int_0^{|f(z)|} \rho(s)ds\le \frac{4 }{\pi}\arctan |z|,$$
$$\rho(|f(z)|)|\nabla f(z)|\le \frac{4}{\pi}\frac{1}{1-|z|^2}$$
and
$$d_\rho(f(z_1),f(z_2)\le \frac{4}{\pi}d_h(z_1,z_2),$$ where $d_h$ is the hyperbolic metric.
\end{corollary}
\begin{proof}[Proof of Corollary~\ref{qpo}]
By Proposition~\ref{lema} in the appendix below, we know that $[-R,R]$ is a geodesic line. Thus Theorem~\ref{copo1} enters on the stage and the results follow.
\end{proof}
\section{Proof of the main result}

Assume that $\gamma$ is a smooth Jordan arc in $ D_{\rho}(c,r)\subset \Omega$ with endpoints $p$ and $q$. Let $\Phi$ be a conformal mapping of $D_\rho(c,r)$  onto the unit disk. Then there are points $p'$ and $q'$ in the  $\partial D_{\rho}(c,r)$ and smooth arcs $\alpha(p,p')$ and $\beta(q,q')$ so that $\Gamma=\gamma+\alpha(p,p')+\mathbf{T}(p',q')+\beta(q,q')$ is a smooth Jordan curve surrounding a Jordan  domain $G\subset D_{\rho}(c,r) $. Let $b$ be a conformal mapping of $G$ onto $\mathbf{H}$, where $\mathbf{H}$ is the half-plane so that $b^{-1}(\infty)\in \mathbf{T}$. Then

\begin{equation}\label{br}b(\gamma)\subset \mathbf{R}.\end{equation}
By taking into account the previous notation, we have the following lemma

\begin{lemma}\label{lemica}
Assume that $f:D\to \gamma\subset D_{\rho}(c,r)$ and assume that $b:G\to \mathbf{H}$ is a conformal mapping. Then the function  $u=b\circ f$ is a real harmonic mapping with respect to the metric $$\tilde\rho(u)=\rho(b^{-1}(u))|(b^{-1})'(u)|$$ defined in a neighborhood of $b(\gamma)\subset {\overline{\mathbf{H}}}$. Moreover Hopf differential of $u$ is equal to the square of a holomorphic function defined in $D$.
\end{lemma}
\begin{proof}
Let us find the Hopf differential of $u$.

We have \[\begin{split}\mathrm{Hopf}(u)&=\tilde \rho^2 (u(z))u_z \overline{u}_zdz^2\\&=\mathrm{Hopf}(b\circ f)\\&=\tilde \rho^2(b\circ f) |b'(f(z))|^2 f_z \bar f_z dz^2\\& =\rho[b^{-1}(b(f(z)))]^2\left|(b^{-1})'(b(f(z)))\right|^2|b'(f(z))|^2 f_z \bar f_z dz^2\\&=\mathrm{Hopf}(f).\end{split}\]

Since $$\tilde \rho(u(z))^2u_z \bar u_z $$ is holomorphic and $u$ is real, it follows that $$\tilde \rho(w(z))^2u_z \bar u_z =(\tilde \rho(u)u_z)^2.$$ Thus $\tilde \rho(u)u_z$ is a holomorphic function defined in $\mathbf{U}$.
\end{proof}

\begin{lemma}
Let  $g:D\to \mathbf{R}$ be a smooth function in an open domain $D\subset \mathbf{C}$. Then for $[z_1,z_2]\subset {D}$ we have $$\int_{[z_1,z_2]}\rho(g(\zeta))du(\zeta)=\int_{g(z_1)}^{g(z_2)}\rho(s) ds.$$
\end{lemma}

\begin{proof}
Let $R$ be so that $R'(s)=\rho(s)$. Then $$\int_{g(z_1)}^{g(z_2)}\rho(s) ds=R(g(z_2))-R(g(z_1)).$$ On the other hand \begin{equation}\begin{split}\int_{[z_1,z_2]}\rho(u(\zeta))du(\zeta)&=\int_{[z_1,z_2]}\rho(u(\zeta))(u_\zeta d\zeta +u_{\bar \zeta} d\bar \zeta)\\&=\int_{z_1}^{z_2} d(R(g(\zeta)))=R(g(z_2))-R(g(z_1)). \end{split}\end{equation}

\end{proof}

\begin{proof}[Proof of Theorem~\ref{copo1}]
First of all we have
$$d_\rho(f(z_1),f(z_2))=\inf_{ \Gamma\ni f(z_1),f(z_2)} |\int_{\Gamma} \rho(w) dw|= \int_{\gamma(f(z_1),u(z_2))} \rho(u) du|$$
where $\gamma(f(z_1),f(z_2))$ is the  part of $\gamma$  between $f(z_1)$ and $f(z_2)$.

Let $\tilde a$ be the antiderivative of $\tilde \rho(u)u_z$, where $\tilde\rho$ is defined in Lemma~\ref{lemica}, i.e.
\begin{equation}\label{tildea}\tilde a(z)=\int \tilde\rho(u(z))u_z dz.\end{equation} Note that $w=\tilde a(z)$ is the so-called distinguished parameter of Hopf differential which is a certain quadratic differential (\cite{strebel}).

Since
$$du=u_x dx + u_y dy $$ and $$u_z=\frac{1}{2}(u_x- i u_y),  \ \  dz=dx+i dy$$ it follows that $$du= 2\Re (u_{ z} dz).$$
Further \[\begin{split}d_\rho(f(z_1),f(z_2))=\int_{\gamma} \rho(z)|dz|&= \int_{b(\gamma)}\rho(b^{-1}(u))|(b^{-1})'(u)| |du|\\&=\left|\int_{b(\gamma)}\rho(b^{-1}(u)) \left|(b^{-1})'(u)\right| du\right|\\&=\left|\int_{b(\gamma)}\tilde\rho(u)du\right|
\\&=\left|\int_{[z_1,z_2]}\tilde\rho(b( f(\zeta)))d(b( f(\zeta)))\right|
.\end{split}\]
Now we have $$\tilde\rho(b( f(\zeta)))d(b( f(\zeta)))= 2\Re\left[\tilde \rho(u(z))u_z dz\right].$$
Therefore by \eqref{tildea} we have $$\int_{\gamma} \rho(\zeta)|d\zeta|=2| \int_{[z_1,z_2]}\Re \left[\tilde a'(\zeta)d\zeta\right] |=2|\Re \int_{[z_1,z_2]}\tilde a'(\zeta)d\zeta |=2|\Re \tilde a(z_2)-\Re \tilde a(z_1)|.$$
So if $f(z)\in D_\rho(f(0),r)$ then $$d_\rho(f(z),f(0))=2|\Re \tilde a(z)-\Re \tilde a(0)|=2|\Re a(z)|.$$

Thus $h(z)=2\Re a(z)$ is a harmonic function defined on the unit disk so that $|h(z)|\le r$ and $h(0)=0$. By the well-known inequality \eqref{schwarz} for real harmonic functions we have
\begin{equation}\label{arca} d_\rho(f(z),f(0))\le \frac{4r}{\pi}\arctan(|z|). \end{equation}
By dividing \eqref{arca} by $|z|$ and using the equation $$d_\rho(f(z),f(0))=\int_{\gamma(f(0), f(z))} \rho(z)|dz|$$ and letting $|z|\to 0$ we get $$\rho(f(0))|Df(0)|\le \frac{\pi r}{4}.$$

In order to prove \eqref{thosh} and \eqref{sos} we do as follows. 
Let $$g(a)=f\left(\frac{z+a}{1+a\bar z}\right).$$ Then $g(0)=f(z)$. So
\begin{equation}\label{sho2}\rho(g(0))|D g(0)|\le \frac{4}{\pi}.\end{equation} Therefore \begin{equation}\label{sho3}\rho(f(z))|Df(z)|(1-|z|^2)\le \frac{4}{\pi} .\end{equation} Thus implies \eqref{thosh}.
By integrating \eqref{thosh} we get
$$d_\rho(f(z_1),f(z_2))\le \frac{4}{\pi}d_h (z_1,z_2).$$
\end{proof}
\begin{example} Let us demonstrate the validity of our result for the following special cases.
a) If $\rho=\frac{1}{1-|z|^2}$ is the Hyperbolic metric, then $$d_\rho(f(z),0)=\tanh^{-1}(|f(z)|).$$
Further $f$ is a real hyperbolic harmonic if and only if $f=\tanh(g)$ where $g$ is real harmonic function (\cite{km}). So $$d_\rho(f(z),0)=\tanh^{-1}(\tanh (g(z)))=|g(z)|=|2\Re a(z)|\le \frac{4}{\pi}\arctan(|z|),$$ provided that $$u(z)\in D_\rho(0,1),$$ i.e. $|g(z)|\le 1$.

b) If $\rho=\frac{1}{1+|z|^2}$ is the Riaemann metric, then $$d_\rho(f(z),0)=\tan^{-1}(|f(z)|).$$ Similarly as before   $f$ is a real Riemann harmonic mapping if and only if $f=\tan(g)$ where $g$ is real harmonic function (\cite{km}). So $$d_\rho(f(z),0)=\tan^{-1}(\tan (g(z)))=|g(z)|=|2\Re a(z)|\le \frac{4}{\pi}\arctan(|z|),$$ provided that $$u(z)\in D_\rho(0,1),$$ i.e. $|g(z)|\le 1$.

\end{example}

\section{Appendix}

\begin{proposition}\label{lema}
If the metric $\rho_\Sigma$ in a chart $D$ of a Riemann surface
$\Sigma$ is given by $\rho_\Sigma(z) = h(|z|^2)$, then the intrinsic
distance of $lz,z\in D$, $l\in\mathbf{R} $, with $[lz,z]\subset D$, is given by
\begin{equation}\label{metrika}d_\Sigma(lz,z) =
\abs{\int_{l|z|}^{|z|}{h(t^2)}dt}.
\end{equation}

In particular, if $z\in D$ and if $[0,z]\in D$ then $[0,z]$ is a
geodesic in $D$ with respect to the metric $\rho_\Sigma$.
\end{proposition}
\begin{proof}

To prove this we do as follows.

Since $g_{11} = g_{22} = h^2(|z|^2)$, and $g_{12}= g_{21}=0$, using
the formula
$$\Gamma^i {}_{k\ell}=\frac{1}{2}g^{im} \left(\frac{\partial
g_{mk}}{\partial x^\ell} + \frac{\partial g_{m\ell}}{\partial x^k} -
\frac{\partial g_{k\ell}}{\partial x^m} \right) = {1 \over 2} g^{im}
(g_{mk,\ell} + g_{m\ell,k} - g_{k\ell,m}), $$

where the matrix $(g^{jk})\ $ is an inverse of the matrix $(g_{jk}\
),$ we obtain that the Christoffel symbols of our metric are given
by:

\begin{equation}\label{g1}\Gamma_{11}^1 =\Gamma_{12}^2 = \Gamma_{21}^2
=\frac{h_x}{h},\end{equation}

\begin{equation}\label{g2} \ \Gamma_{22}^2 = \Gamma_{12}^1 =
\Gamma_{21}^1=\frac{h_y}{h},\end{equation}

\begin{equation}\label{g3}\Gamma_{11}^2 = -\frac{h_x}{h},\ \  \Gamma_{22}^1 =
-\frac{h_y}{h}.\end{equation}

The geodesic equations are given by:

    $$\frac{d^2x^\lambda }{ds^2} + \Gamma^{\lambda}_{~\mu \nu }\frac{dx^\mu }{ds}\frac{dx^\nu }{ds} = 0\
,\lambda = 1, 2.
    $$
In view of \eqref{g1}, \eqref{g2} and \eqref{g3} we obtain the
system:

\begin{equation}\label{11} \ddot x + 2\frac{xh'}{h}{\dot x}^2 + 4\frac{yh'}{h}\dot x \dot
y - 2\frac{xh'}{h}{\dot y}^2 = 0,\end{equation}

\begin{equation}\label{22} \ddot y - 2\frac{yh'}{h}{\dot x}^2 + 4\frac{xh'}{h}\dot x \dot
y + 2\frac{yh'}{h}{\dot y}^2=0.\end{equation}

Assume, first that $0<l<1$ and $[l|z|, |z|]\subset D$. Denote the geodesic curve
joining the points $|z|$ and $|z|l$ by $c(s) := (x(s), y(s))$.

Due to uniqueness property of geodesic, we try to find the (uniques) solution by having in mind the constraint $y = 0$. By plugging $y=0$  in \eqref{11} and \eqref{22} we obtain that $x$ is a
solution of the differential equality

$$\ddot x + 2\frac{xh'}{h}{\dot x}^2=0$$ and consequently

$$\dot x = \frac{C_1}{h(x^2)},$$ i.e.

\begin{equation}\label{sx1}s = C_1\int_{x_0}^{x}{h(t^2)}dt.\end{equation}

To determine $C_1$ and $x_0$, we use the conditions $x(0) = l|z|$,
and $x(s_0) = |z|$. Inserting these conditions to \eqref{sx1} we
obtain

\begin{equation}\label{sx}s = \int_{l|z|}^{x}{h(t^2)}dt,
\end{equation}
where $$s_0 =\int_{l|z|}^{|z|}{h(t^2)}dt.$$

 As the metric $h(|z|^2)|dz|$ is a rotation invariant,
according to \eqref{sx} it follows that $$d_\Sigma(lz,z) =
\inf_{lz,z\in
\gamma}\int_{\gamma}\rho_\Sigma(z)|dz|=\int_{l|z|}^{|z|}h(r^2)dr.$$
The other cases can be reduced to this case.
\end{proof}

\end{document}